\documentclass[12pt]{amsart}

\usepackage{amssymb}
\usepackage{amsmath}
\usepackage{amsfonts}
\usepackage{graphicx}
\usepackage{amscd}
\usepackage{xcolor}
\usepackage{float}
\usepackage{mathtools}

\newtheorem{thm}{Theorem}[section]
\newtheorem{cor}[thm]{Corollary}
\newtheorem{lema}[thm]{Lemma}
\newtheorem{prop}[thm]{Proposition}
\theoremstyle{definition}
\newtheorem{defn}[thm]{Definition}
\theoremstyle{remark}
\newtheorem{rem}[thm]{Remark}

\def\R{\mathbb{R} }

\newcommand{\N}{\mathbb N}

\def\to{\rightarrow}
\def\supp{\mathop{\text{\normalfont supp}}}

\begin{document}

\title{Peridynamics and Anisotropic Fractional Sobolev Spaces with Variable Exponents}

\author{Sabri Bahrouni}
\address[S. Bahrouni]{ Mathematics Department, Faculty of Sciences, University of Monastir, 5019 Monastir, Tunisia\\
The author is supported by FAPESP Proc 2023/04515-7}
\email{sabri.bahrouni@fsm.rnu.tn}

 \author[J. Fern\'{a}ndez Bonder]{Juli\'{a}n Fern\'{a}ndez Bonder}
\address[J. Fern\'{a}ndez Bonder]{Instituto de C\'alculo, CONICET,Departamento de Matem\'atica, FCEN - Universidad de Buenos Aires, Ciudad Universitaria, 0+$\infty$ building, C1428EGA, Av. Cantilo s/n, Buenos Aires, Argentina}
\email{jfbonder@dm.uba.ar}

\author[I. Ceresa Dussel]{Ignacio Ceresa Dussel}
\address[I. Ceresa Dussel]{Instituto de C\'alculo, CONICET,Departamento de Matem\'atica, FCEN - Universidad de Buenos Aires, Ciudad Universitaria, 0+$\infty$ building, C1428EGA, Av. Cantilo s/n, Buenos Aires, Argentina}
\email{iceresad@dm.uba.ar}

\author[O. Miyagaki]{Olimpio Miyagaki}
\address[O. Miyagaki]{Departamento de Matem\'{a}tica, Universidade Federal de S\~ao Carlos S\~ao Carlos, SP, CEP:13565-905,  Brazil\\
The author is supported by CNPq Proc 303256/2022-2 and FAPESP Proc 2022/16407-1}
\email{ olimpio@ufscar.br.}

\begin{abstract}
In this paper, our primary objective is to develop the peridynamic fractional Sobolev space and establish novel BBM-type results associated with it. We also address the peridynamic fractional anisotropic $p-$Laplacian. A secondary objective is to explore anisotropic fractional Sobolev spaces with variable exponents, where we also derive new BBM-type results. Additionally, we address the eigenvalue problem in the isotropic case.
\end{abstract}


\maketitle

\section{Introduction}

Given a measurable function $u$ on $\mathbb{R}^N$, and an index vector $\overrightarrow{p}=\left(p_1, p_2, \ldots, p_N\right)$, where $1\leq p_j \leq \infty$ for each $j$, we can calculate the numbers $\|u\|_{\overrightarrow{p}}$ by first calculating the $L^{p_1}$-norm of $u\left(x_1, \ldots, x_N\right)$ with respect to $x_1$, and then the $L^{p_2}$-norm of the result with respect to $x_2$, and so on finishing with the $L^{p_N}$-norm with respect to $x_N$:
$$
\|u\|_{\overrightarrow{p}}=\left\|\cdots\left\|\|u\|_{L^{p_1}\left(dx_1\right)}\right\|_{L^{p_2}\left(d x_2\right)} \cdots\right\|_{L^{p_N}\left(d x_N\right)}
$$
where
$$
\|f\|_{L^p(d t)}=\left\{\begin{array}{l}
{\left(\displaystyle\int_{-\infty}^{\infty}|f(\ldots, t, \ldots)|^p d t\right)^{1 / p} \text { if } 1<p<\infty} \\
\underset{t}{\operatorname{ess} \sup }|f(\ldots, t, \ldots)| \text { if } p=\infty .
\end{array}\right.
$$

If all the numbers $p_j$ are finite, then
$$
\|u\|_{\overrightarrow{p}}=\left(\int_{-\infty}^{\infty} \ldots\left(\int_{-\infty}^{\infty}\left(\int_{-\infty}^{\infty}\left|u\left(x_1, \ldots, x_n\right)\right|^{p_1} \mathrm{~d} x_1\right)^{p_2 / p_1} \mathrm{~d} x_2\right)^{p_3 / p_2} \ldots \mathrm{d} x_N\right)^{1 / p_N} .
$$

We will denote by $L^{\overrightarrow{p}}=L^{\overrightarrow{p}}\left(\mathbb{R}^N\right)$ the Banach space of set of functions $u$ for which $\|u\|_{\overrightarrow{p}}<\infty,$ where equivalence classes of almost everywhere equal. We refer to Benedek and Panzone \cite{Benedek-Panzone} for general information on spaces $L^{\overrightarrow{p}}$.

We define the anisotropic Sobolev space as the class of function $u$ in $L^{\overrightarrow{p}}$ whose partial derivatives $\partial_{i}u$  are elements of $L^{p_i},$ see \cite{Adams} and references therein for more details.

This type of anisotropy led to two possible generalizations: one for Sobolev spaces with a variable exponent and the other for fractional Sobolev spaces.

For Sobolev spaces with a variable exponent, variational problems have been extensively studied, with a particular focus on the isotropic variable exponent Sobolev spaces, as detailed in \cite{DHHR, Radulescu-Repovs1, Fan-Zhao, Kova}. In \cite{Fan1}, Fan carried out a systematic investigation of anisotropic variable exponent Sobolev spaces, providing an important theoretical framework for studying anisotropic and parabolic equations involving the operator $-\Delta_{\overrightarrow{p}(x)}u:=-\sum_{i=1}^{N}\partial_i\left(\left|\partial_i u\right|^{p_i(x)-2}\partial_i u\right),$ (see \cite{Antontsev-Shmarev1, Antontsev-Shmarev2, MPR}).

While variable exponent Sobolev spaces extend the classical theory by allowing the integrability order to vary, the fractional framework further generalizes this concept by using non-integer orders of differentiability, leading to a richer and more nuanced analysis of function spaces. Fractional order Sobolev spaces have gained significant attention recently, particularly following the influential papers \cite{NPV} and \cite{Caffarelli-Silvestre}. In the seminal paper \cite{BBM1}, Bourgain, Brezis and Mironescu proved that for any smooth bounded domain $\Omega \subset \mathbb{R}^N, u \in W^{1, p}(\Omega), 1 \leq p<\infty$, the well-known fractional Gagliardo seminorm recovers its local counterpart as $s$ goes to 1, in the sense that
$$
\lim _{s \uparrow 1}(1-s) \iint_{\Omega \times \Omega} \frac{|u(x)-u(y)|^p}{|x-y|^{N+s p}} d x d y=K(N, p) \int_{\Omega}|\nabla u|^p d x,$$
 where the constant $K(N, p)$ is defined as
$$
K(N, p)=\frac{1}{p} \int_{\mathbb{S}^{N-1}}|\boldsymbol{\omega} \cdot h|^p d \mathcal{H}^{N-1}(h).
$$
Here $\mathbb{S}^{N-1} \subset \mathbb{R}^N$ denotes the unit sphere, $\mathcal{H}^{N-1}$ is the ($N-1)$-dimensional Hausdorff measure and $\omega$ is an arbitrary unit vector of $\mathbb{R}^N$. See \cite{BBM2} for a survey on this and related results.

In \cite{CKW1}, the authors studied the fractional anisotropic Sobolev space, where they considered different fractional regularity and integrability in each coordinate direction. They introduced $N $ different fractional parameters $s_1, \cdots, s_N$ and $N$ different integrability parameters $p_1, \cdots , p_N$ , defining a space of functions $u(x) = u(x_1, \cdots , x_N) $ such that
$$
\int_{\R^N}\int_{\R}\frac{|u(x+h e_i)-u(x)|^{p_i}}{|h|^{1+s_i p_i}}\, dhdx<\infty\quad\text{for}\quad i=1,\cdots, N.
$$
A Bourgain-Brezis-Mironescu (BBM) type result for an energy functional strongly related to the fractional anisotropic $p-$Laplacian has been recently proven in \cite{Ceresa-Bonder1}.\\

One of the primary goals of this paper is, based on \cite{Bellido-Mora, Bellido-Ortega}, to study the peridynamic problem associated with the fractional anisotropic $p$-Laplacian. Peridynamics is distinguished by its nonlocality, which allows points separated by a positive distance to exert forces on each other. This feature sets peridynamics apart from classical theories that rely on gradients. As a result, peridynamics is particularly effective for addressing problems involving discontinuities, such as fractures, dislocations, or multi-scale materials. Hence, for a given real number $\delta>0$, $0<s<1$, and $1<p<\infty$, we will analyze the behavior when $\delta\to 0$ of the functional of the form
\begin{equation}
	[u]_{s,p,\delta}^i=\int_{\R^N}\int_{|h|\leq \delta}\frac{|u(x+he_i)-u(x)|^p}{|h|^{1+sp}}\,dh\,dx,
\end{equation}
 and the associated nonlocal operator, namely
$$
(-\Delta_p)_{\delta}^s u(x)= p.v.\int_{|h|\leq \delta}\frac{|u(x+he_i)-u(x)|^{p-2}(u(x+he_i)-u(x))}{|h|^{1+sp}}\,dh.
$$
It can be seen as a \emph{ peridynamic fractional anisotropic p-laplacian} because $(-\Delta_p)_{\delta}^s $ represents a truncation of the fractional anisotropic p-laplacian analyzed in the paper \cite{Ceresa-Bonder1}.\\

Recently, a novel category of fractional Sobolev spaces with variable exponents, denoted as $W^{s,\bar{p},p(.,.)}(\Omega)$, was introduced by U. Kaufmann {\it et al.} \cite{Kaufmann}. Initially proposed in their work, this framework has subsequently undergone refinement in \cite{Ho-Kim}, where it is precisely defined as
$$
W^{s,\bar{p},p(.,.)}(\Omega)=\bigg{\{}u\in L^{\bar{p}}(\Omega):\ \int_{\Omega}\int_{\Omega}\frac{|u(x)-u(y)|^{p(x,y)}}{|x-y|^{N+sp(x,y)}}dxdy<\infty\bigg{\}},
$$
where $\Omega$  is an open subset of $\mathbb{R}^N$ $(N\geq2)$, $s\in(0,1)$, $\bar{p}(x)=p(x,x)$ and $p$ satisfying
\begin{equation}\tag{P}\label{P}
  \begin{aligned}
  &p\in C\left(\mathbb{R}^N\times\mathbb{R}^N,(1+\infty)\right),\\
  &p(x,y)=p(y,x)\ \text{for all}\ x,y\in\mathbb{R}^N.
\end{aligned}
\end{equation}

A second goal of this paper is to explore the anisotropic properties of fractional Sobolev spaces with variable exponents, which will be discussed in Section \ref{Anisotropic fractional Variable Exponents}.

\begin{defn}
Given $i=1,\cdots, N,$ $p_0\in C\left(\mathbb{R}^N,(1+\infty)\right),$ $\Omega$ an open subset of $\R^N,$ and $p_i$ satisfying \eqref{P}.
Let $\vec{s}=(s_1,\cdots,s_N),$ and $\vec{p}(\cdot,\cdot)=(p_1(\cdot,\cdot),\cdots,p_N(\cdot,\cdot)).$ We introduce the {\bf{\it anisotropic fractional Sobolev spaces with variable exponent}} $W^{\vec{s},\vec{p}(\cdot,\cdot)}_{p_0}(\Omega)$ as
$$
W^{\vec{s},\vec{p}(\cdot,\cdot)}_{p_0}(\R^N)=\left\{u\in L^{p_0}(\R^N):\quad\sum_{i=1}^{N}J_{s_i,p_i}(u)<\infty\right\},
$$
equipped with the norm
$$
\|u\|_{\vec{s},p_0,\vec{p}}=\|u\|_{p_0}+\sum_{i=1}^{N}[u]_{s_i,p_i},
$$
where
$$
J_{s_i,p_i}(u)=\int_{\R^N}\int_{\R}\frac{|u(x+h e_i)-u(x)|^{p_i(x,x+h e_i)}}{|h|^{1+s_i p_i(x,x+h e_i)}}\, dhdx,
$$
 and
$$
[u]_{s_i,p_i}=\inf\left\{\lambda>0:\quad J_{s_i,p_i}\left(\frac{u}{\lambda}\right))\leq1\right\}.
$$
\end{defn}

The rest of the paper is organized as follows: Section 2 is a preliminary section to establish the notation. In Section 3, we address the Peridynamics case. In Section 4, we cover the anisotropic variable case. Note that in the last section, we sometimes do not specify the exponent as variable in the notations if it is clear from the context.

\section{Preliminaries}

In this section, we briefly review the definitions and list some basic properties of the
Lebesgue spaces with variable exponent and the fractional Sobolev spaces with variable exponent.

\subsection{Variable exponent Lebesgue spaces}

Let $\Omega$ be a Lipschitz domain in $\mathbb{R}^N$ and $r\in C_+(\overline{\Omega})$. We define the variable exponent Lebesgue space
$$
L^{r(\cdot)}(\Omega):=\bigg{\{}u:\Omega\rightarrow\mathbb{R},\ \text{measurable}\ \text{and}\ \int_{\Omega}|u(x)|^{r(x)}dx<\infty\bigg{\}},
$$
endowed with the Luxemburg norm
$$
\|u\|_{r,\Omega}:=\inf\bigg{\{}\lambda>0:\  \int_{\Omega}\left|\frac{u(x)}{\lambda}\right|^{r(x)}\frac{dx}{r(x)}\leq1 \bigg{\}}.
$$
In the Luxemburg norm definition, we used $\frac{dx}{r(x)},$ recognizing that this choice is inconsequential; it remains equivalent to using $dx$ alone and streamlines the equations slightly. When there is no confusion we shall often write $\|\cdot\|_{r}$ instead of $\|\cdot\|_{r,\Omega}.$

Some basic properties of $L^{r(x)}(\Omega)$ are listed in the following two propositions.

\begin{prop}[\cite{Fan-Zhao, Kova}]
  The space $L^{r(\cdot)}(\Omega)$ is a separable, uniformly convex Banach space and its dual space is $L^{r^{'}(\cdot)}(\Omega)$, where $\frac{1}{r(x)}+\frac{1}{r^{'}(x)}=1$. Furthermore, for any $u\in L^{r(\cdot)}(\Omega)$ and $v\in L^{r^{'}(x)}(\Omega)$, we have $$\bigg{|}\int_{\Omega}uv\ dx\bigg{|}\leq 2\|u\|_{r,\Omega}\|v\|_{r^{'},\Omega}.$$
\end{prop}

\begin{prop}[\cite{Fan-Zhao}]
  Define the modular $\rho_r:L^{r(\cdot)}(\Omega)\rightarrow\mathbb{R}$ by $$\rho_r(u):=\int_{\Omega}|u|^{r(x)}dx,\ \text{for all}\ u\in L^{r(\cdot)}(\Omega).$$
  Then, we have the following relations between norm and modular.\\
  \begin{enumerate}
    \item [(i)] $u\in L^{r(\cdot)}(\Omega)\setminus\{0\}$ if and only if $\rho_r\big{(}\frac{u}{\|u\|_{r,\Omega}}\big{)}=1$.
    \item [(ii)] $\rho_r(u)>1\ (=1,\ <1)$  if and only if $\|u\|_{r,\Omega}>1,\ (=1,\ <1)$, respectively.
    \item [(iii)] If $\|u\|_{r,\Omega}>1$, then $\|u\|_{r,\Omega}^{r_-}\leq\rho_r(u)\leq \|u\|_{r,\Omega}^{r_+}$.
    \item [(iv)]  If $\|u\|_{r,\Omega}<1$, then $\|u\|_{r,\Omega}^{r_+}\leq\rho_r(u)\leq \|u\|_{r,\Omega}^{r_-}$.
    \item [(v)] For a sequence $(u_n)\subset L^{r(\cdot)}(\Omega)$ and $u\in L^{r(x)}(\Omega)$, we have $$\lim_{n\rightarrow+\infty}\|u_n-u\|_{r,\Omega}=0\ \ \text{if and only if}\ \ \lim_{n\rightarrow+\infty}\rho_r(u_n-u)=0.$$
  \end{enumerate}
\end{prop}

\subsection{Fractional Sobolev spaces with variable exponents}

Let $s\in(0,1)$ and $p\in C(\overline{\Omega}\times\overline{\Omega})$  be such that $p$ is symmetric, i.e., $p(x,y) = p(y,x)$ for all $x,y\in \overline{\Omega}$ and
 $$
 1<p^-(\Omega):=\inf_{(x,y)\in\overline{\Omega}\times\overline{\Omega}}p(x,y)\leq p^+(\Omega):=\sup_{(x,y)\in\overline{\Omega}\times\overline{\Omega}}p(x,y)<+\infty.
 $$

Let $\bar{p}(x)=p(x,x),$ then $\bar{p}\in C_+(\overline{\Omega})$, and we define
$$
W^{s,\bar{p},p(.,.)}(\Omega)=\bigg{\{}u\in L^{\bar{p}(x)}(\Omega):\ \int_{\Omega}\int_{\Omega}\frac{|u(x)-u(y)|^{p(x,y)}}{|x-y|^{N+sp(x,y)}}dxdy<\infty\bigg{\}},
$$
and for $u\in W^{s,\bar{p},p(.,.)}(\Omega)$, set
$$
[u]_{s,p(.,.),\Omega}=\inf\bigg{\{}\lambda>0:\ \rho_{p(.,.)}\bigg{(}\frac{u}{\lambda}\bigg{)} <1\bigg{\}},
$$
where
$$
\rho_{p(.,.)}(u)=\int_{\Omega}\int_{\Omega}\frac{|u(x)-u(y)|^{p(x,y)}}{|x-y|^{N+s\tilde{p}(x,y)}}\frac{dxdy}{p(x,y)}.
$$
Then, $W^{s,\bar{p},p(.,.)}(\Omega)$  endowed with the norm
$$
\|u\|_{s,\bar{p},p,\Omega}:=\|u\|_{\bar{p},\Omega}+[u]_{s,p(.,.),\Omega},
$$
is a separable and reflexive Banach space (see \cite{Kaufmann, anouar1, anouar2}).

\begin{thm}\label{embedding in frac}(\cite{Ho-Kim})
Assume that $sp^+(\Omega)<N.$
Then, it holds that
\begin{enumerate}
  \item the embedding $W^{s,\bar{p},p(.,.)}(\mathbb{R}^N)\hookrightarrow L^{r(x)}(\mathbb{R}^N)$ is continuous for any $r\in C_+(\mathbb{R}^N)$ satisfying $p(x)\leq r(x)\ll p^{*}_s:=\frac{Np(x,x)}{N-sp(x,x)}$ for all $x\in \mathbb{R}^N$.
  \item if $\Omega$ is bounded, then the embedding $W^{s,p(.),p(.,.)}(\Omega)\hookrightarrow\hookrightarrow L^{r(x)}(\Omega)$ is compact  for any $r\in C_+(\Omega)$ satisfying $r(x)<p^{*}_s$ for all $x\in \overline{\Omega}$.
\end{enumerate}
\end{thm}

Invoking the continuity of $p$ on $\bar{\Omega} \times \bar{\Omega}$ we extend $p$ to $\mathbb{R}^N \times \mathbb{R}^N$ by using Tietze extension theorem, such that
\begin{equation}\label{P}\tag{P}
  1<\inf _{(x, y) \in \mathbb{R}^N \times \mathbb{R}^N} p(x, y) \leq \sup _{(x, y) \in \mathbb{R}^N \times \mathbb{R}^N} p(x, y)<\frac{N}{s}.
\end{equation}
 We now define the following space:
$$
X=\left\{u \in W^{s,\bar{p}, p(\cdot, \cdot)}\left(\mathbb{R}^N\right): \quad u=0 \text { on } \Omega^c\right\}
$$
endowed with norm
$$
\|u\|_X=[u]_{s, p(\cdot, \cdot), \mathbb{R}^N},
$$
where $W^{s,\bar{p}, p(\cdot, \cdot)}\left(\mathbb{R}^N\right)$ and $[u]_{s, p(\cdot, \cdot) \mathbb{R}^N}$ are defined in the same ways as $W^{s,\bar{p}, p(\cdot, \cdot)}(\Omega)$ and $[u]_{s, p(\cdot, \cdot), \Omega}$ with $\Omega$ replaced by $\mathbb{R}^N$. Obviously, $X$ is a closed subspace of $W^{s,\bar{p}, p(\cdot, \cdot)}\left(\mathbb{R}^N\right)$ and hence, $\left(X,\|\cdot\|_X\right)$ is a reflexive and separable Banach space.

\begin{lema}
  The functional $\rho_{p(.,.)}: X \rightarrow \mathbb{R}$ defined by
$$
\rho_{p(.,.)}(u)=\int_{\mathbb{R}^N} \int_{\mathbb{R}^N} \frac{|u(x)-u(y)|^{p(x, y)}}{|x-y|^{N+s p(x, y)}} \frac{dxdy}{p(x,y)},
$$
has the following properties:
\begin{itemize}
  \item [(i)] for $\alpha>0,\|u\|_X=(>,<) \alpha$ if and only if $\rho_{p(.,.)}\left(\frac{u}{\alpha}\right)=(>,<) 1$,
  \item [(ii)] if $\|u\|_X>1$, then $\|u\|_X^{p^{-}} \leq \rho_{p(.,.)}(u) \leq\|u\|_X^{p^{+}}$,
  \item [(iii)] if $\|u\|_X<1$, then $\|u\|_X^{p^{+}} \leq \rho_{p(.,.)}(u) \leq\|u\|_X^{p^{-}}$.
\end{itemize}
\end{lema}

\begin{thm}[\cite{Bahrouni-Ho}]\label{compact-embedding}
Let $\Omega \subset \mathbb{R}^N$ be a bounded Lipschitz domain and let $s \in(0,1)$. Let $p \in$ $C\left(\mathbb{R}^N \times \mathbb{R}^N\right)$ satisfy \eqref{P}. Then, for any $r \in C(\bar{\Omega})$ satisfying
\begin{equation}\label{R}\tag{R}
  r(x)<p^{*}_s\quad \text{for all}\quad x\in \overline{\Omega},
\end{equation}
there exists a constant $C>0$ such that
\begin{equation}\label{Sobolev-inequality}
  \|u\|_{r,\Omega}\leq C \|u\|_X,\quad \forall\ u\in X.
\end{equation}
  Moreover, the embedding $X\hookrightarrow L^{r(.)}(\Omega)$ is compact.
\end{thm}



\section{Peridynamics}
In this section we will developed the theory for the peridynamic fractional anisotropic Sobolev space and the \textit{peridynamic fractional anisotropic p-laplacian}.
\subsection{Peridynamic fractional anisotropic Sobolev space}
 Given a real number $\delta>0$, $0<s<1$ and $1<p<\infty$. We will analyze, the behavior as $\delta\to 0$ of functional of the form
\begin{equation}\label{seminorm-i}
	[u]_{s,p,\delta}^i=\int_{\R^n}\int_{|h|\leq \delta}\frac{|u(x+he_i)-u(x)|^p}{|h|^{1+sp}}\,dh\,dx.
\end{equation}
First, we need to define an appropriate functional space to work within; therefore, we introduce the anisotropic peridynamical Sobolev space as follows:
\begin{equation}
    W^{s,p,\delta}_i(\R^n)\colon=\{u \in L^p(\R^n), \text{ such that }[u]_{s,p,\delta}^i<\infty\}.
\end{equation}
This space is endowed with the norm
$$
\|u\|_{s,p,\delta}^i=\left([u]_{s,p,\delta}^i\right)^{\frac{1}{p}}+\|u\|_{p},
$$
is a reflexive and separeble Banach space.\\
 Now, in order to control every variable, we proceed as follows. Let $\vec{s}=(s_1,\dots, s_n)$, $\vec{p}=(p_1,\dots, p_n)$ and $\vec{\delta}=\{\delta_1,\dots,\delta_n\}$ with $s_i\in (0,1]$, $p_i\in (1,\infty)$ and $\delta_i>0$ for $i=1,\dots, n$. Then define
$$
W^{\vec{s},\vec{p},\vec{\delta} }(\R^n)=\bigcap_{i=1}^n W^{s_i,p_i,\delta_i}_i(\R^n).
$$
This anisotropic Sobolev space is then a separable and reflexive Banach space with norm
$$
\|u\|_{\vec{s},\vec{p},\vec{\delta}} = \sum_{i=1}^n \|u\|_{s_i,p_i,\delta_i}^i.
$$
Observe that the case when each $\delta_i$ tends to infinity corresponds to the Sobolev space studied in \cite{CKW1}. In that paper, the authors define the quantities
$$
\overline{\vec{s}}=\left(\frac{1}{n}\sum_{i=1}^n\frac{1}{s_i}\right)^{-1},  \qquad  \overline{\vec{s}\vec{p}}=\left(\frac{1}{n}\sum_{i=1}^n\frac{1}{s_ip_i}\right)^{-1}.
$$
Then, assuming that $\overline{\vec{s}\vec{p}} < n$ the following critical Sobolev exponent is introduced
\begin{equation}
\vec{p}_{\vec{s}}\ ^* = \vec{p}\ ^*=\frac{n\overline{\vec{s}\vec{p}}/\overline{\vec{s}}}{n-\overline{\vec{s}\vec{p}}}.
\end{equation}
The following theorem is a straightforward adaptation of Theorem 1.1 of \cite{CKW1} and the proof is omitted.
\begin{thm}
    Let $ s_1, \ldots, s_n \in [s_0, 1) $ for some $s_0 \in (0, 1) $ and $ p_1, \ldots, p_n > 1 $ be such that $\overline{\vec{s}\vec{p}} < n$. Then $ W^{\vec{s}, \vec{p},\vec{\delta}}(\R^n) \subset\subset L^p_{\text{loc}}(\R^n) $ for every $ 1 \leq p < \vec{p}\ ^* $. Moreover, $ W^{\vec{s}, \vec{p},\vec{\delta}}(\R^n) \subset L^p_{\text{loc}}(\R^n) $  continuously for every $ 1 \leq p \leq  \vec{p}\ ^* $.
\end{thm}

\subsection{BBM-type results for $W_i^{s,p,\delta}(\R^n)$}
Let us start by revising several classical lemmas to establish the BBM-type results.

In the rest of this section we will use the following notation $D_i(u)(x)=|u(x+he_i)-u(x)|.$

\begin{lema}\label{lemma1}
	Let $u\in W^{1,p}_i(\R^N)$. Therefore the following estimate holds
	$$[u]_{s,p,\delta}^i\leq  2\frac{\delta^{p(1-s)}}{p(1-s)}\left\|\partial_i u\right\|_{p}^{p}.$$
\end{lema}

\begin{proof}
	Given $u \in C^1_c(\R^N)$ it is a well known fact that
	$$\int_{\R^N}|D_i(u)(x)|^{p}\, dx\leq |h|^{p} \|\partial_i u\|_{p}^{p},
	$$
	then
	\begin{align*}
		\int_{|h|<\delta}\int_{\R^N}\frac{|D_i(u,x)|^{p}}{|h|^{1+sp}}\, dx dh =& \int_{|h|<\delta}\frac{1}{|h|^{1+sp}}\left(\int_{\R^N}|D_i(u)(x)|^{p}\, dx\right)\, dh
		\\&\leq \left\|\partial_i u\right\|_{p}^{p}2\int_{0}^{\delta}\frac{h^p}{|h|^{1+sp}}\, dh\\
		&=2\frac{\delta^{p(1-s)}}{p(1-s)}\left\|\partial_i u\right\|_{p}^{p}.
	\end{align*}
	
	We conclude the proof using the density of smooth functions with compact support in $W^{1,p}_i(\R^N)$.
\end{proof}

The next lemma shows that the functional $[\cdot]_{s,p,\delta}^i$ decreases when one regularizes a function $u$ with a standard mollifier and shows the behavior with respect to truncations.  The proof is in \cite[Lemma 2.3]{Ceresa-Bonder1}.

\begin{lema}\label{regular}
On one hand, let $u \in L^p(\R^N)$, $\rho_\varepsilon$ is the standard mollifier, then $u_\epsilon=\rho_\varepsilon * u $ verifies
	$$
	[u_\varepsilon]_{s,p,\delta}^i \leq [u]_{s,p,\delta}^i
	$$
for every $ \varepsilon >0$.

On the other hand, given $\eta\in C^1_c(\R^N)$ be such that $\eta =1$ on $B_1$, $\eta=0$ on $B_2^c$, $0\le \eta\le 1$ and $\|\nabla \eta\|_\infty\le 2$. Given $k\in\N$, we define $\eta_k(x)=\eta(x/k)$ . If  $u_k=\eta_ku$ is the truncation of $u$ by $\eta_k$, then
	$$
	[u_k]_{s,p,\delta}^i \leq   2^{p-1} [u]_{s,p,\delta}^i +\frac{2^{2p}\delta^{p(1-s)}}{k^pp(1-s)}\|u\|_p^p.
	$$
\end{lema}

The following theorems are the main important in this section.
\begin{thm}\label{Peridynamics-teo1}
	Let $u\in C^2_c(\R^N)$, $0<s<1,$ and $\delta >0$. Then
	\begin{equation}\label{eq1p}
	\lim_{\delta\to0}\frac{1}{\delta^{p(1-s)}}[u]_{s,p,\delta}^i=\frac{2}{p(1-s)}\|\partial_i u\|_p.
	\end{equation}
\end{thm}

\begin{proof}
	From the well-known estimate for smooth functions
	\begin{equation}\label{c2estimate}
	\left|\frac{|D_i(u)(x)|^p}{|h|^p}-|\partial_i u|^p\right|\leq C|h|,
	\end{equation}
	where $C$ depends of the smoothness of $u$, we get that
	\begin{align*}
\int_{|h|<\delta}\left|\frac{|D_i(u)(x)|^p}{|h|^{1+sp}}-\frac{|\partial_i u|^p}{|h|^{1+p(s-1)}}\right|\,dh\leq \frac{2C\delta^{p(1-s)+1}}{p(1-s)+1}.
	\end{align*}
	Hence, when $\delta \to 0$ we conclude that
	\begin{align*}
	\lim_{\delta\to 0}\int_{|h|<\delta}\frac{|D_i(u)(x)|^p}{|h|^{1+sp}} \,dh=\lim_{\delta\to 0}\int_{|h|<\delta}\frac{|\partial_i u|^p}{|h|^{1+p(s-1)}}\,dh=\lim_{\delta\to 0}2|\partial_i u|^p\frac{\delta^{p(1-s)}}{p(1-s)}.
	\end{align*}
	Therefore
	\begin{align*}
	\lim_{\delta\to 0}\frac{1}{\delta^{p(1-s)}}\int_{|h|<\delta}\frac{|D_i(u)(x)|^p}{|h|^{1+sp}} \,dh=\frac{2|\partial_i u|^p}{p(1-s)}.
	\end{align*}
To end the proof we need to find an integrable majorant for
$$
F_\delta(x)=\int_{|h|<\delta}\frac{|D_i(u)(x)|^p}{|h|^{1+sp}} \,dh.
$$
Then since $u\in  C^2_0(\R^N) $, there exists $R>\delta>0$ such that $\supp(u)\subset Q_R$, where $Q_R=[-R,R]^n$. If $|x_i|\leq 2R$
	\begin{equation}
|F_\delta(x)|=\int_{|h|<\delta}\frac{|u(x+he_i)-u(x)|^p}{|h|^{1+sp}}\,dh\leq \|\partial_{i}u||_{\infty}^p \frac{\delta^{p(1-s)}}{p(1-s)},
	\end{equation}
	$$
	\left|\frac{1}{\delta^{p(s-1)}}F_\delta(x)\right|\leq C \chi_{Q_{2R}}(x),
	$$
	where  C is a constant depending of $p,R$ and $s$. Finally if $|x_i|> 2R$, $F_\delta(x)=0$.
Hence by Lebesgue's Dominated Convergence Theorem
\begin{equation}
\lim_{\delta\to0}\frac{1}{\delta^{p(1-s)}}[u]_{s,p,\delta}^i=\frac{2}{p(1-s)}\|\partial_i u\|_p
\end{equation}
for all $u \in C^2_c(\R^N).$
\end{proof}

\begin{rem}\label{rem1}
	From equation \eqref{c2estimate} the convergence in \eqref{eq1} is uniform on bounded subset of $C^2(\overline{\Omega}_\varepsilon)$.
\end{rem}

\begin{thm}\label{peridynamics-teo3}
	Let $u\in L^p(\R^N)$, $0<s<1,$ and $\delta >0$. Then
	$$
	\lim_{\delta\to0}\frac{1}{\delta^{p(1-s)}}[u]_{s,p,\delta}^i=\frac{2}{p(1-s)}\|\partial_i u\|_p.
	$$
\end{thm}

\begin{proof}
	 On the one hand, let $u \in W^{1,p}_i(\R^N)$, $s \in (0,1),$ $p \in (1,\infty),$ and $\delta>0$ . Take a sequence $\{u_k\}_{k\in \N} \subset C^2_0(\R^N)$ such that $u_k\to u$ in $L^p(\R^N)$ and $\partial_i u_k\to \partial_i u$ in $L^p(\R^N)$. Then

	\begin{align*}
		\left| \frac{1}{\delta^{p(1-s)}}[u]_{s,p,\delta}^i - \frac{2}{p(1-s)}\|\partial_i u\|_{p}\right| & \leq
    \delta^{(s-1)} \left| [u]_{s,p,\delta}^i - [u_k]_{s,p,\delta}^i\right|\\
		& + \left| \frac{1}{\delta^{p(1-s)}}[u_k]_{s,p,\delta}^i-\frac{2}{p(1-s)}\|\partial_i u_k\|_{p}\right|\\
		& + \frac{2}{p(1-s)} \left|\|(u_k)_{x_i}\|_{p} - \|u_{x_i}\|_{p}\right|\\
		&= I  + II  + III.
	\end{align*}
	
	Note that $[\cdot ]_{s,p,\delta}^i$ is a seminorm, which means it satisfies the triangle inequality. Therefore by Lemma \ref{lemma1}, we have
	$$
	I\leq\delta^{(s-1)}[u_k-u]_{s,p,\delta}^i\leq C(p,s)\|\partial_i u_k -\partial_i u\|_p,
	$$
	which converges to zero if $k\to \infty$, uniformly in $\delta$. The third term is bounded similarly, and the second term follows as a consequence of Theorem \ref{Peridynamics-teo1}.
	
	 On the other hand, given $u \in L^{p}(\R^N)$ and assume that
	\begin{equation}\label{liminf1}
		\liminf_{\delta \to 0}\frac{1}{\delta^{p(1-s)}}[u]_{s,p,\delta}^i < \infty.
	\end{equation}
	
	Let $u_{k,\varepsilon}\subset C^\infty_0(\R^N)$ be given by
	$$
	u_{k,\epsilon}=\rho_\varepsilon*(u\eta_k),
	$$
	and  applying Lemma \ref{regular} together with \eqref{liminf1}, we obtain

	$$
	\liminf_{\delta \to 0}\frac{1}{\delta^{p(1-s)}}[u_{k,\varepsilon}]_{s,p,\delta}^i < C,
	$$
	where $C$ is independent from $k$ and $\varepsilon>0$.
	
	Next, for any fixed $k$ and $\epsilon$, we apply Theorem \ref{Peridynamics-teo1} and obtain
	\begin{equation}
		C\geq 	\lim_{\delta \to 0}\frac{1}{\delta^{p(1-s)}}[u_{k,\varepsilon}]_{s,p,\delta}^i  = \frac{2}{p(1-s)}\|\partial_i u_{k,\varepsilon}\|_p.\nonumber
	\end{equation}
	
	Hence, there exist a subsequence $u_j=\{u_{k_j,\epsilon_j}\}\subset \{u_{k,\varepsilon}\}$ such that $\partial_{i} u_j \rightharpoonup \partial_{i} u$ weakly in $L^p(\R^N)$.
	Thus $u \in W_i^{1,p}(\R^N)$. This completes the proof of the Theorem.
\end{proof}

Finally, we are interested in proving the sequence case.
\begin{thm}\label{Peridynamics-teo2}
	Let $\delta_k\to0$ and $\{u_k\}_{k\in \N}\subset L^p(\R^N)$ is a sequence such that
	$$ \sup_k\|u_k\|_p< \infty,\quad u_k\to u \text{ in }L_{loc}^p(\R^N) \text{ and } \sup_{k}[u_k]_{s,p,\delta_k}<\infty.$$
	Then
	$$
	\frac{2}{p(1-s)}\|\partial_i u\|_p\leq \liminf_{\delta_k\to0}\frac{1}{\delta_k^{p(1-s)}}[u_k]_{s,p,\delta_k}^i.
	$$
\end{thm}

\begin{proof}
Based on Remark \ref{rem1}, we can apply Ponce's method \cite{Ponce}. Let $\varepsilon>0$ and $R>0$ be fixed. From Lemma \ref{regular} we get that $$[u_k]_{s,p,\delta_k}^i\geq[u_{k,\varepsilon}]_{s,p,\delta_k}^i\geq \int_{B_R(0)}\int_{|h|\leq \delta_k}\frac{|u_{k,\varepsilon}(x+he_i)-u_{k,\varepsilon}(x)|^p}{|h|^{1+sp}}\,dh\,dx.
 $$
 For each $\varepsilon$ fixed we have that $u_{k,\varepsilon}\to u_\varepsilon$ in $C^2(B_R(0))$. From Remark \ref{rem1} we get that
 \begin{align*}
 \frac{1}{\delta_k^{p(1-s)}}\int_{B_R(0)}\int_{|h|\leq \delta_k}\frac{|u_{k,\varepsilon}(x+he_i)-u_{k,\varepsilon}(x)|^p}{|h|^{1+sp}}\,dh\,dx\to\frac{2}{p(1-s)}\int_{B_R(0)}|\partial_i u_\varepsilon|^p\,dx
 \end{align*}
 as $k\to\infty$.
Therefore,
$$
\liminf_{\delta_k\to0}\frac{1}{\delta_k^{p(1-s)}}[u_k]_{s,p,\delta_k}^i\geq \frac{2}{p(1-s)}\int_{B_R(0)} |(\partial_i u)_\varepsilon|^p\,dx.
$$
Taking $\varepsilon\to0$ and $R\to\infty$, we get the desired result.
\end{proof}

\subsection{Applications}
In most applications, what turns out to be more useful than the norm or semi-norms are the energy functionals
$$J_{\delta}^i(u)\colon= \frac{1}{\delta^{p(1-s)}}[u]_{s,p,\delta}^i \text{ and } J^i(u)\colon= \frac{2}{p(1-s)}\|u_{x_i}\|_p^p.$$
Observe that we avoid using the notation $ J_{s,p,\delta}^i(u) $ because it would be cumbersome. However, the reader should understand that the functional also depends on $p$ and $s$.

Theorems \ref{peridynamics-teo3} and \ref{Peridynamics-teo2} immediately imply
$$
J_{\delta}^i(u)\overset{\Gamma}{\longrightarrow}  J_{0}^i(u)
$$
as $\delta\to 0$ in $L^p(\Omega)$ for every $\Omega\subset \R^N$ bounded. Just observe that Theorem \ref{Peridynamics-teo2} is exactly the lim-inf inequality and for the
lim sup inequality one just has to take the constant sequence $u_k = u$ as a recovery
sequence and apply Theorem \ref{peridynamics-teo3}.

Now, given $\vec{p}=(p_1,\dots,p_n)$ and $\vec{s}=(s_1,\dots,s_n)$, $s_i \in (0,1)$ and $\vec{\delta}=(\delta_1,\dots,\delta_n)$  We define the functional
$$
J_{\vec{\delta}}(u) = \sum_{i=1}^n J_{\delta_i}^i(u).
$$
In the context of the application, establishing the Gamma-convergence of the entire functional is crucial. The following corollary ensures this.

\begin{cor}\label{gammaperidynamic}
Let $\vec{p}=(p_1,\dots,p_n)$ and $\vec{s}=(s_1,\dots,s_n)$ with $p_i\in (1,\infty)$ and $s_i \in (0,1)$ for every $i=1,\dots, n$. Let $\delta_i^k\to 0$ as $k\to\infty$ for $i=1,\dots,j$ and let $\overline{\delta}_k=(\delta_1^k,\dots,\delta_j^k,\delta_{j+1},\dots,\delta_n)$. Let $\vec{\delta}_0 = (\underbrace{0,\dots,0}_{j}, \delta_{j+1},\dots,\delta_n)$. Then $J_{\vec{\delta}_k}$ Gamma converges to $J_{\vec{\delta}_0}$ in  $L^{p_\text{max}}(\Omega)$ for any $\Omega\subset\R^n$ bounded.
\end{cor}
\begin{proof}
The $\liminf$ inequality for $J_{\vec{\delta_k}}$ follows from the $\liminf$ inequality for every $J^i_{\delta_i^k}$ $(i=1,\dots,j)$ and the continuity of $J^i_{\delta_i}$ for $i=j+1,\dots,n$.

The $\limsup$ inequality follows since the constant sequence $u_k=u$ is a recovery sequence for {\em every} $J^i_{\delta_i^k}$, $i=1,\dots,j$.
\end{proof}

\subsubsection{Peridynamic fractional anisotropic p-laplacian}
In this subsection, we will show the variational nature of the problem by defining the Euler-Lagrange equation for the functionals.

\begin{defn}
Let $\Omega\subset \R^n$ be a bounded open subset. Then we define the anisotropic Sobolev space of functions vanishing at the boundary as
\begin{align*}
W^{\vec{s},\vec{p},\vec{\delta}}_0(\Omega)&=\{u\in W^{\vec{s},\vec{p},\vec{\delta}}(\R^n) \colon u|_{\Omega^c}=0\},\\
W^{s_i,p_i,\delta_i}_{i,0}(\Omega)&=\{u\in W^{s_i,p_i,\delta_i}_{i}(\R^n) \colon u|_{\Omega^c}=0\}.
\end{align*}
The topological dual spaces of $W^{\vec{s},\vec{p},\vec{\delta}}_0(\Omega)$ and $W^{s_i,p_i,\delta_i}_{i,0}(\Omega)$  will be denoted by $W^{-\vec{s},\vec{p}',\vec{\delta}}(\Omega) $ and $W_i^{-s_i,p'_i,\delta_i}(\Omega)$ respectively.
\end{defn}

Now, it is easy to see that the functionals $J^i_{\delta_i}$ are Fréchet differentiable for every $i=1,\dots,n$, every $s\in (0,1]$, $\delta_i>0$ and every $p\in (1,\infty)$. Even more  $(J^i_{\delta_i})'\colon W^{s_i,p_i,\delta_i}_{i,0}(\Omega) \to W^{-s_i,p_i',\delta_i}_i(\Omega)$ is continuous and  is given by
\begin{align*}
\langle (J^i_{\delta_i})'(u),v\rangle=
\frac{p}{\delta^{p(1-s)}}\int_{\R^n}\int_{|h|\leq \delta}\frac{|D_i(u)(x)|^{p-2}D_i(u)(x)D_i(v)(x)}{|h|^{1+sp}}\,dh dx.
\end{align*}

$$
\langle (J^i)'(u), v\rangle= \frac{2}{(1-s)}\int_{\R^n}|\partial_{x_i}u|^{p-2}\partial_{x_i}u \ \partial_{x_i}v \,dx,
$$
where $\langle\cdot, \cdot\rangle$ denotes the duality pairing between $W^{s_i,p_i,\delta}_{i,0}(\Omega)$ and its dual $W^{-s_i,p_i',\delta}_i(\Omega)$.

The operator $$
J_{\vec{\delta}}(u) = \sum_{i=1}^n J_{\delta_i}^i(u).
$$
is Fr\'echet differentiable with derivative given by
$$
$$
$$
J_{\vec{\delta}}'(u) = \sum_{i=1}^n \frac{1}{p_i}(J^i_{\delta_i})'(u).
$$

 Finally, we define the \textit{Peridynamic fractional anisotropic $\vec{p}$-laplacian} as
 $$
 (-\widetilde{\Delta}_{\vec{p}})^{\vec{s}}_{\vec{\delta}} u  := (J_{\vec{\delta}})'(u).
 $$

Therefore, we want to study the equation
\begin{equation}\label{EDP-anisotropica}
\begin{cases}
(-\widetilde{\Delta}_{\vec{p}})^{\vec{s}}_{\vec{\delta}} u = f & \text{in }\Omega,\\
u=0 & \text{in } \R^n\setminus \Omega.
\end{cases}
\end{equation}
We say that $u \in W^{\vec{s},\vec{p}}_0(\Omega)$ is a weak solution of \eqref{EDP-anisotropica} if
$$
\langle(-\widetilde{\Delta}_{\vec{p}})^{\vec{s}}_{\vec{\delta}}u,v\rangle=\int_\Omega f v\,dx
$$
for all $v \in W^{\vec{s},\vec{p},\vec{\delta}}_0(\Omega)$.

For a weak solution to be well defined, as usual, we need to impose some integrability conditions on the source term $f$. Hence, we consider the case where $f=f(x)$ and $f\in L^{(\vec{p}\ ^*)'}(\Omega)$.

The existence and uniqueness of weak solutions to \eqref{EDP-anisotropica} is then a direct consequence of the direct method in the calculus of variations since the solution is the unique minimizer of the functional
$$
I(v) := J_{\vec{\delta}}(v) - \int_\Omega fv\, dx,
$$
which is a strictly convex, coercive, and continuous functional in $W^{\vec{s},\vec{p},\vec{\delta}}_0(\Omega)$.

Let us summarize all of these facts into a single statement.
\begin{prop}\label{existence1}
Let $\Omega\subset \R^n$ be a bounded open set, let $\vec{\delta}=(\delta_1,\dots,\delta_n)$, $\delta_i>0$,  $\vec{s}=(s_1,\dots,s_n)$, $s_i\in (0,1)$, $\vec{p}=(p_1,\dots,p_n)$, $p_i\in (1,\infty)$ and  $\vec{p}\ ^*$ the critical Sobolev exponent defined in the beginning of this section. Assume that $p_\text{max}<\vec{p}\ ^*$ and $\overline{\vec{s}\vec{p}}<n$.

Then, for every $f\in L^{(\vec{p}\ ^*)'}(\Omega)$, there exists a unique weak solution $u\in W^{\vec{s},\vec{p},\vec{\delta}}_0(\Omega)$
of \eqref{EDP-anisotropica}.
\end{prop}

Now, we will analyze the asymptotic behavior of the solution to \eqref{EDP-anisotropica} when (or some of the) $\delta_i\to 0$.

To this end assume that $\delta_1^k,\dots, \delta_j^k\to 0$ (as $k\to\infty$) and $\delta_{j+1},\dots,\delta_n$ remain fixed and consider $\vec{\delta}_k := (\delta_1^k,\dots,\delta_j^k,\delta_{j+1},\dots,\delta_n)$ and the functionals
$$
I_k(v) := J_{\vec{\delta}_k}(v) - \int_\Omega fv\, dx.
$$
By Proposition \ref{existence1}, there exists a unique minimizer $u_k$ of $I_k$ in $W^{\vec{s},\vec{p},\vec{\delta}}_0(\Omega)$. Hence we want to study the behavior of the sequence $\{u_k\}_{k\in\N}$ as $k\to\infty$.

To this end, recall that $\vec{\delta}_0 := (\underbrace{0,\dots,0}_{j},\delta_{j+1},\dots,\delta_n)$ and
$$
I_0(v) := J_{\vec{\delta}_0}(v) - \int_\Omega fv\, dx,
$$
where
$$
J_{\vec{\delta}_0}(v) := \sum_{i=1}^j J^i(v) + \sum_{i=j+1}^n J^i_{\delta_i}(v).
$$

Next, we make full use of the BBM results to conclude the following theorem:
\begin{thm}
With the preceding notation, the functionals $\bar I_k\colon L^{p_\text{min}}(\Omega)\to \bar\R$ defined as
$$
\bar I_k(v) =\begin{cases}
I_k(v) & \text{if } v\in W^{\vec{s},\vec{p},\vec{\delta}}_0(\Omega),\\
\infty & \text{elsewhere}.
\end{cases}
$$
Gamma-converges to $\bar I\colon L^{p_\text{min}}(\Omega)\to \bar\R$ defined as
$$
\bar I(v) =\begin{cases}
I_0(v) & \text{if } v\in W^{\vec{s},\vec{p},\vec{\delta}}_0(\Omega),\\
\infty & \text{elsewhere}.
\end{cases}
$$
\end{thm}

\begin{proof}
By Corollary \ref{gammaperidynamic} we have that $J_{\vec{\delta}_k}$ gamma converge to $J_{\vec{\delta}_0}$ as $k\to\infty$. So the result follows since the functional $u\mapsto \int_\Omega fu\, dx$ is continuous in $L^{p_\text{min}}(\Omega)$ .
\end{proof}

An immediate consequence of this result is the convergence of the solutions of \eqref{EDP-anisotropica} to the corresponding limiting problem.
\begin{cor}\label{fixed f}
Let $u_k\in W^{\vec{s},\vec{p},\vec{\delta}}_0(\Omega)$ be a weak solution to
$$
\begin{cases}
(-\widetilde{\Delta}_{\vec{p}})^{\vec{s}}_{\vec{\delta}_k} u_k = f & \text{in }\Omega,\\
u_k=0 & \text{in } \R^n\setminus\Omega.
\end{cases}
$$
Then $u_k\to u_0$ in $L^{p_\text{min}}(\Omega)$ where $u_0\in W^{\vec{s},\vec{p},\vec{\delta_0}}_0(\Omega)$ is the weak solution to
$$
\begin{cases}
(-\widetilde{\Delta}_{\vec{p}})^{\vec{s}_0}_{\vec{\delta}_k} u_0 = f & \text{in }\Omega,\\
u_0=0 & \text{in } \R^n\setminus\Omega.
\end{cases}
$$
\end{cor}

\section{Anisotropic fractional Sobolev Spaces with Variable Exponents}\label{Anisotropic fractional Variable Exponents}

\subsection{Properties and Inequalities}

We define the fractional $i^{th}-$Sobolev variable exponent space as
$$
W_i^{s_i,p_i(\cdot,\cdot)}(\R^N):=\left\{u\in L^{\bar{p}_i}(\R^N):\quad J_{s_i,p_i(\cdot,\cdot)}(u)<\infty\right\}.
$$
We can see without difficulty that the space $W_i^{s_i,p_i(\cdot,\cdot)}(\mathbb{R}^N)$ constitutes a Banach space, characterized by the norm
$$
\|u\|_{s_i,p_i}=\|u\|_{\bar{p}_i}+[u]_{s_i,p_i},
$$
which is both separable and reflexive.
Define
$$
W^{\vec{s},\vec{p}(\cdot,\cdot)}(\R^N)=\bigcap_{i=1}^N W_i^{s_i,p_i(\cdot,\cdot)}(\R^N),
$$
then $W^{\vec{s},\vec{p}(\cdot,\cdot)}(\R^N)$ is a Banach, separable and reflexive space.

\begin{defn}
  A function $u\in L^{1}_{loc}(\R^N)$ belongs to the {\bf{\it homogeneous anisotropic fractional Sobolev spaces with variable exponent }} $\dot{W}^{\vec{s},\vec{p}(\cdot,\cdot)}(\R^N)$ if $\sum_{i=1}^{N}[u]_{s_i,p_i}<\infty.$
\end{defn}

\begin{prop}\label{inclusion}
  If $u\in C^1_c(\R^N),$ then $u\in \dot{W}^{\vec{s},\vec{p}(\cdot,\cdot)}(\R^N).$
\end{prop}

\begin{proof}
Let $u\in C^1_c(\R^N),$ there exists $R>0$ such that $u(x)=0$ for all $x\in \R^N,$ $|x|\geq R.$ By the mean value Theorem,
$$
|u(x)-u(y)|\leq \|\nabla u\|_\infty |x-y|,\quad\forall\ x,y\in B_R(0).
$$
By Tonelli's theorem and relabeling the variables of integration, we can write
\begin{align*}
  &\int_{\R^N}\int_{\R}\frac{|u(x+he_i)-u(x)|^{p_i(x,x+he_i)}}{|h|^{1+s_i p_i(x,x_i(h))}}\, dhdx\\ & =\int_{B_R(0)}\int_{I_i}\frac{|u(x+he_i)-u(x)|^{p_i(x,x+he_i)}}{|h|^{1+s_i p_i(x,x+he_i)}}\, dhdx\\
  &\leq \int_{B_R(0)}\int_{I_i}\frac{\|\nabla u\|_\infty^{p_i(x,x+he_i)} |h|^{p_i(x,x+he_i)}}{|h|^{1+s_i p_i(x,x+he_i)}}dhdx\\
  &\leq \max\left(\|\nabla u\|_\infty^{p^-},\|\nabla u\|_\infty^{p^+}\right)
  \int_{I_i}\frac{\max\left(|h|^{p^-},|h|^{p^+}\right)}{\min\left\{|h|^{1+s_i (p_i)^{-})},|h|^{1+s_i (p_i)^+)}\right\}}dh\\
  &:=C_i
\end{align*}
where $C_i>0,$ and $I_i=\left\{h\in\R :\ \ x+he_i\in B_R(0),\ \forall x\in B_R(x)\right\}.$ It is not difficult to see that $I_i$ will be a bounded subset of $\R.$ This completes the proof.
\end{proof}

\begin{defn}
  We define $W^{\vec{s},\vec{p}(\cdot,\cdot)}_{\bar{p}}(\Omega)$ as a $C^1_c(\mathbb{R}^N)$-module if, for any $u \in W^{\vec{s},\vec{p}(\cdot,\cdot)}_{\bar{p}}(\Omega)$ and $v \in C^1_c(\mathbb{R}^N)$, the product $uv$ also belongs to $W^{\vec{s},\vec{p}(\cdot,\cdot)}_{\bar{p}}(\Omega)$.
\end{defn}

Possessing the $C^1_c(\mathbb{R}^N)$-module property is considered advantageous. Let's investigate the conditions under which $W^{\vec{s},\vec{p}(\cdot,\cdot)}_{\bar{p}}(\Omega)$ is indeed a $C^1_c(\mathbb{R}^N)$-module.

\begin{prop}
  Let $\Omega$ be a bounded subset of $\R^N.$ $W^{\vec{s},\vec{p}(\cdot,\cdot)}_{\bar{p}}(\Omega)$ is a $C^1_c(\mathbb{R}^N)$-module if
  $$
  \bar{p}(x)\geq p_M(x)\quad\text{for a.e.}\ x\in\Omega,
  $$
  where $p_M(x):=\max\left\{p_1(x,x),\cdots,p_N(x,x)\right\}.$
\end{prop}

\begin{proof}
  Since $p_i(x)\leq \bar{p}(x),$ then $L^{\bar{p}}(\Omega)\hookrightarrow L^{p_i(\cdot)}(\Omega).$ Let $u\in W^{\vec{s},\vec{p}(\cdot,\cdot)}_{\bar{p}}(\Omega)\subset L^{\bar{p}}(\Omega)$ and $v\in C^1_c(\mathbb{R}^N)\subset L^{p_0(\cdot)}(\Omega)$, then, obviously $uv\in L^{p_0(\cdot)}(\Omega).$ From Proposition \ref{inclusion},  $uv\in \dot{W}^{\vec{s},\vec{p}(\cdot,\cdot)}(\R^N).$ Thus $W^{\vec{s},\vec{p}(\cdot,\cdot)}_{p_0(\cdot)}(\Omega)$ is a $C^1_c(\mathbb{R}^N)$-module.
\end{proof}

The space $W^{\vec{s},\vec{p}(\cdot,\cdot)}_{p_M(\cdot)}(\Omega)$ is written simply by $W^{\vec{s},\vec{p}(\cdot,\cdot)}(\Omega),$ namely
$$
W^{\vec{s},\vec{p}(\cdot,\cdot)}(\Omega)=\left\{u\in L^{p_M(\cdot)}(\Omega):\quad\sum_{i=1}^{N}J_{s_i,p_i,\Omega}(u)<\infty\right\}.
$$

\subsection{BBM-type results for $W_i^{s,p(\cdot,\cdot)}(\R^N)$}

In this section, we are operating with a constant direction and fixed parameters. Consequently, we will dispense with the use of the index "$i$" and simply represent them as $s,p$ and $\overline{p}$ instead of $s_i,p_i$ and $\overline{p}_i$ respectively, and we fix ourselves in the first direction.

\begin{thm}\label{BBM}
  Let $u\in C^2_c(\R^N).$ Then
  $$
  \lim_{s\to 1} J_{s,p(\cdot,\cdot)}=\int_{\R^N}\frac{2}{\overline{p}(x)}|\partial_{1}u|^{\bar{p}(x)}\,dx.
  $$
\end{thm}

\begin{lema}\label{BBMinf}
   Let $u\in C^2_c(\R^N).$ Then
  \begin{equation}\label{BBMliminf}
    \int_{\R^N}\frac{2}{\overline{p}(x)}|\partial_1u|^{\bar{p}(x)}dx\leq \liminf_{s\nearrow1}(1-s)\int_{\R^N}\int_{\R}\frac{|u(x+he_1)-u(x)|^{p(x,x+he_1)}}{|h|^{1+s p(x,x+he_1)}}\, dhdx.
  \end{equation}
\end{lema}

\begin{proof}
  Since $u\in C^2_c(\R^N),$ there exists $C>0,$ depending on the $C^2-$norm of $u,$ such that
  \begin{equation}\label{BBMinf1}
    |u(x+he_1)-u(x)-\partial_{1}u(x).h|\leq C|h|^2,\quad\text{for all}\quad h\in\R.
  \end{equation}
 Consider fixing $\varepsilon > 0$ and selecting $x \in \mathbb{R}^N.$ The continuity of the function $h \mapsto p(x, x + he_1)$ at $0$ implies the existence of a positive value $r_1,$ this value satisfies the condition that
  \begin{equation}\label{continuity of p}
    (1-\varepsilon)|\partial_{1}u(x)|^{\bar{p}(x)}\leq |\partial_{1}u(x)|^{p(x,x+he_1)}\leq (1+\varepsilon)|\partial_{1}u(x)|^{\bar{p}(x)},
  \end{equation}
  for all $|h|<r_1.$
   Also we have the existence of $r_2>0,$ such that
   \begin{equation}\label{continuity of hp}
    (1-\varepsilon)|h|^{\bar{p}(x)}\leq |h|^{p(x,x+he_1)}\leq (1+\varepsilon)|h|^{\bar{p}(x)},
  \end{equation}
   for all $|h|<r_2.$ Again there exists $r_3>0,$ such that
   \begin{equation}\label{continuity of hhp}
    (1-\varepsilon)\frac{1}{|h|^{\bar{p}(x)}}\leq \frac{1}{|h|^{p(x,x+he_1)}}\leq (1+\varepsilon)\frac{1}{|h|^{\bar{p}(x)}},
  \end{equation}
   for all $0<|h|<r_3.$
   We can derive from \eqref{BBMinf1},
  $$
  |\partial_{1}u(x).h|\leq  |u(x+he_1)-u(x)|+ C|h|^2.
  $$
  Using the inequality
  $$
  (a+b)^q\leq (1+\varepsilon)a^q+C(q,\varepsilon)b^q,\quad a,b\geq 0,\ q>1,
  $$
  we get
  $$
  |\partial_{1}u(x).h|^{p(x,x+he_1)}\leq (1+\varepsilon)|u(x+he_1)-u(x)|^{p(x,x+he_1)}+C |h|^{2p(x,x+he_1)}.
  $$
  If we consider $r_0=\min\{1,r_1,r_2,r_3\}$ and take $0<r<r_0,$ then
  \begin{align*}
    &\int_{|h|<r}\frac{|\partial_{1}u(x).h|^{p(x,x+he_1)}}{|h|^{1+sp(x,x+he_1)}}dh\\
    &\leq
  (1+\varepsilon)\int_{|h|<r}\frac{|u(x+he_1)-u(x)|^{p(x,x+he_1)}}{|h|^{1+sp(x,x+he_1)}}dh
  +C\int_{|h|<r}\frac{|h|^{2p(x,x+he_1)}}{|h|^{1+sp(x,x+he_1)}}dh .
  \end{align*}
 Let's estimate from below the left-hand side of the previous inequality. Referring to \eqref{continuity of p}, \eqref{continuity of hp} and \eqref{continuity of hhp} we find that
  \begin{equation}\label{frombelow}
  \begin{aligned}
     \int_{|h|<r}\frac{|\partial_{1}u(x).h|^{p(x,x+he_1)}}{|h|^{1+sp(x,x+he_1)}}dh
  &\geq(1-\varepsilon)^3|\partial_{1}u|^{\bar{p}(x)}\int_{|h|<r}|h|^{-1+(1-s)\overline{p}(x)}dh\\
  &=\frac{(1-\varepsilon)^3}{(1-s)\overline{p}(x)}|\partial_{1}u|^{\bar{p}(x)}r^{(1-s)\overline{p}(x)}.
  \end{aligned}
  \end{equation}

  Therefore, integrating over $\R^N,$ we get
  \begin{align*}
     &\frac{2}{p^-}(1-\varepsilon)r^{(1-s)p^-}\int_{\R^N}|\partial_{1}u|^{\bar{p}(x)}dx\\
     &\leq (1+\varepsilon)(1-s)\int_{\R^N}\int_{\R}\frac{|u(x+he_1)-u(x)|^{p(x,x+he_1)}}{|h|^{1+sp(x,x+he_1)}}dhdx
     +C\frac{2(1-s)}{(2-s)p^-}r^{(2-s)p^-}.
  \end{align*}
  We pass to $\liminf_{s\nearrow1},$ we obtain
  \begin{align*}
     &\frac{2}{p^-}(1-\varepsilon)r^{(1-s)p^-}\int_{\R^N}|\partial_{1}u|^{\bar{p}(x)}dx\\
     &\leq (1+\varepsilon)\liminf_{s\nearrow1}(1-s)\int_{\R^N}\int_{\R}\frac{|u(x+he_1)-u(x)|^{p(x,x+he_1)}}{|h|^{1+sp(x,x+he_1)}}dhdx.
  \end{align*}
  Since $\varepsilon$ is arbitrarily chosen, we deduce \eqref{BBMliminf}.
\end{proof}

\begin{lema}\label{BBMsup}
   Let $u\in C^2_c(\R^N).$ Then
  \begin{equation}\label{BBMlimsup}
    \limsup_{s\nearrow1}(1-s)\int_{\R^N}\int_{\R}\frac{|u(x+he_1)-u(x)|^{p(x,x+he_1)}}{|h|^{1+s p(x,x+he_1)}}\, dhdx\leq\int_{\R^N}\frac{2}{\overline{p}(x)}|\partial_{1}u|^{\bar{p}(x)}dx.
  \end{equation}
\end{lema}

\begin{proof}
  The proof of this lemma follows by combining the Dominated Convergence Theorem with the two subsequent lemmas.
\end{proof}

\begin{lema}\label{Dominate}
   Let $u\in C^1_c(\R^N).$ Then there exists a function $F\in L^1(\R^N)$ such that
   $$
   F_s(x):=(1-s)\int_{\R}\frac{|u(x+he_1)-u(x)|^{p(x,x+he_1)}}{|h|^{1+s p(x,x+he_1)}}\, dh\leq F(x),\quad\text{for all}\ x\in \R^N.
   $$
\end{lema}

\begin{proof}
  Since $u\in C^1_c(\R^N),$ there exists $R>1$ such that $supp(u)\subset Q_R,$ where $Q_R=[-R,R]^N.$ If $|x_1|\leq 2R,$
  \begin{align*}
    |F_s(u)| & =(1-s)\left(\int_{|h|\leq 1}+\int_{|h|>1}\right)\frac{|u(x+he_1)-u(x)|^{p(x,x+he_1)}}{|h|^{1+s p(x,x+he_1)}}\, dh:=I_1+I_2.
  \end{align*}
  \begin{align*}
    |I_1| &\leq (1-s)\int_{|h|\leq 1}\|\partial_{1}u\|_{\infty}^{p(x,x+he_1)}|h|^{-1+(1-s)p(x,x+he_1)}dh\\
    &\leq (1-s)\left[\|\partial_{1}u\|_{\infty}^{p^-}+\|\partial_{1}u\|_{\infty}^{p^+}\right]\int_{|h|\leq 1}|h|^{-1+(1-s)p^-}dh\\
    &=\frac{\alpha_N}{p^-}\left[\|\partial_{1}u\|_{\infty}^{p^-}+\|\partial_{1}u\|_{\infty}^{p^+}\right],
  \end{align*}
  where $\alpha_N=|B(0,1)|.$
  We used in the previous inequality
  $$
  |u(x+he_1))-u(x)|\leq \int_{0}^{1}|\partial_{1}u(x+the_1).he_1|dt\leq \|\partial_{1}u\|_{\infty}|h|.
  $$
  Also
  \begin{align*}
   |I_2|  &\leq 2\left[\|u\|_{\infty}^{p^-}+\|u\|_{\infty}^{p^+}\right]\int_{|h|> 1}|h|^{-1-sp^-}dh\\
   &=\frac{2}{sp^-}\left[\|\partial_{1}u\|_{\infty}^{p^-}+\|\partial_{1}u\|_{\infty}^{p^+}\right].
  \end{align*}
  If $|x_1|>2R,$ we denote by $x=(x_1,x^{'})$ with $x^{'}\in\R^{N-1},$ then
  $$
  F_s(x)=(1-s)\int_{\R}\frac{|u(x_1+h,x')-u(x)|^{p(x,(x_1+h,x'))}}{|h|^{1+s p(x,(x_1+h,x'))}}\, dh.
  $$
  Since $|x_1|>2R,$ we use that $|h|\geq |x_1|-|x_1+h|\geq |x_1|-R\geq \frac{1}{2}|x_1|,$ we get
  $$
  | F_s(x)|\leq \left(\frac{2}{|x_1|}\right)^{1+sp^-}\int_{\R}|u(x_1,x')|^{p_1(x,(x_1,x'))}dx_1.
  $$
  Observe that
  $$
  G(x'):=\int_{\R}|u(x_1,x')|^{p_1(x,(x_1,x'))}dx_1\in L^1(\R^{N-1}).
  $$
  Therefore
  $$
   | F_s(x)|\leq C\left(\chi_{Q_{2R}}(x)+\chi_{\{|x_1|\geq 2R\}}(x)\left(\frac{2}{|x_1|}\right)^{1+sp^-}G(x')\right)\in L^1(\R^N),
  $$
  where $C$ is a constant depending on $N,p^-,u.$
\end{proof}

\begin{lema}\label{Dominate-convergence}
  Let $u\in C^2_c(\R^N).$ Then
   $$
    \lim_{s\nearrow1} F_s(x):=\frac{2}{\overline{p}(x)}|\partial_{1}u|^{\bar{p}(x)}.
   $$
\end{lema}

\begin{proof}
  Let $u$ belongs to the space $C^2_c(\mathbb{R}^N)$. In this case, there exists a positive constant $C$ that depends on the norm $|u|_{C^2(\mathbb{R}^N)}$, satisfying the condition:

  \begin{align*}
     |u(x+he_1)-u(x)|^{p(x,x+he_1)}\leq (1+\varepsilon)|\partial_{1}u|^{p(x,x+he_1)}|h|^{p_1(x,x+he_1)}+ C|h|^{2p(x,x+he_1)}
  \end{align*}

  for all $x\in\R^N$ and $h\in\R.$ From the preceding inequality, we obtain
  \begin{align*}
    A(h)&:=\left|\frac{ |u(x+he_1)-u(x)|^{p(x,x+he_1)}}{|h|^{1+sp(x,x+he_1)}}-(1+\varepsilon)|\partial_{1}u|^{p(x,x+he_1)}|h|^{-1+(1-s)p(x,x+he_1)}\right|\\
    &\leq C|h|^{-1+(2-s)p(x,x+he_1)}.
  \end{align*}
 By  performing the integration over $\R$, we arrive at the following expression:
  \begin{equation}\label{h-leq1+h-geq1}
    \int_{\R}A(h)\,dh=\left(\int_{|h|\leq1}+\int_{|h|>1}\right)A(h)\,dh.
  \end{equation}
 In the initial integral on the right side, we observe
  \begin{align*}
    \int_{|h|\leq1}A(h)\,dh &\leq C  \int_{|h|\leq1}|h|^{-1+(2-s)p(x,x+he_1)}\,dh\\
    &\leq C(1+\varepsilon)\int_{|h|\leq1}|h|^{-1+(2-s)\overline{p}(x)}\,dh=\frac{2C}{(2-s)\overline{p}(x)},
  \end{align*}
  then
  $$
  \lim_{s\nearrow1}(1-s)\int_{|h|\leq1}A(h)\,dh=0,
  $$
  that is

  \begin{equation}\label{h-leq1}
  \begin{aligned}
    &\lim_{s\nearrow1}(1-s)\int_{|h|\leq1}\frac{ |u(x+he_1)-u(x)|^{p(x,x+he_1)}}{|h|^{1+sp(x,x+he_1)}}\,dh\\
    &=\lim_{s\nearrow1}(1-s)(1+\varepsilon)|\partial_{1}u|^{p(x,x+he_1)}\int_{|h|\leq1}|h|^{-1+(1-s)p(x,x+he_1)}\,dh\\
    &\leq \lim_{s\nearrow1}2(1-s)(1+\varepsilon)|\partial_{1}u|^{p(x,x+he_1)}\int_{|h|\leq1}|h|^{-1+(1-s)\overline{p}(x)}\,dh\\
    &=(1+\varepsilon)\frac{2}{\overline{p}(x)}|\partial_{1}u|^{p(x,x+he_1)}.
  \end{aligned}
  \end{equation}

  For in the subsequent integral on the right side of \eqref{h-leq1+h-geq1}, we find
  \begin{align*}
    \int_{|h|>1}A(h)\,dh & \leq C2^{p^+}\max\left(\|u\|_{\infty}^{p^-},\|u\|_{\infty}^{p^+}\right)\int_{|h|>1}\frac{dh}{|h|^{1+sp(x,x+he_1)}}\\
    &=\frac{C2^{p^+}}{sp^-}\max\left(\|u\|_{\infty}^{p^-},\|u\|_{\infty}^{p^+}\right),
  \end{align*}
  then
  $$
  \lim_{s\nearrow1}(1-s)\int_{|h|>1}A(h)\,dh=0,
  $$
   and as illustrated in \eqref{h-leq1}, we can see that

  \begin{equation}\label{h-geq1}
  \begin{aligned}
  \lim_{s\nearrow1}(1-s)\int_{|h|\leq1}\frac{ |u(x+he_1)-u(x)|^{p(x,x+he_1)}}{|h|^{1+sp(x,x+he_1)}}\,dh\leq(1+\varepsilon)\frac{2}{\overline{p}(x)}|\partial_{1}u|^{p(x,x+he_1)}.
  \end{aligned}
  \end{equation}

  Therefore, by combining \eqref{h-leq1} and \eqref{h-geq1} and since $\varepsilon$ is arbitrarily, we derive the following:
  \begin{equation}
    \lim_{s\nearrow1}(1-s)F_s(x)\leq \frac{2}{\overline{p}(x)}|\partial_{1}u|^{p(x,x+he_1)}.
  \end{equation}
  Upon examining \eqref{frombelow}, we have also
  \begin{equation}
    \lim_{s\nearrow1}(1-s)F_s(x)\geq \frac{2}{\overline{p}(x)}|\partial_{1}u|^{p(x,x+he_1)},
  \end{equation}
  and the proof is completed.
\end{proof}

\begin{proof}[{\bf Proof of Theorem \ref{BBM}}]
  The proof can be derived by combining both Lemma \ref{BBMinf} and Lemma \ref{BBMsup}.
\end{proof}

\subsection{Eigenvalue problems: Isotropic case}

The fractional $p(x)$-Laplacian $(-\triangle)^s_{p(.,.)}$ is given by
$$
(-\triangle)^s_{p(.,.)}u(x):=P.V.\int_{\mathbb{R}^N}\frac{|u(x)-u(y)|^{p(x,y)-2}(u(x)-u(y))}{|x-y|^{N+sp(x,y)}}dy,\ \text{for all}\ x\in\Omega.
$$
$P.V.$ is a commonly used abbreviation in the principal value sense.\\

The eigenvalue problem associated to $(-\triangle)^s_{p(.,.)}$ is the following:
\begin{equation}\label{classic-eq1}
\begin{cases}
  (-\triangle)^s_{p(.,.)}u=\Lambda |u|^{\bar{p}-2}u & \mbox{in } \Omega \\
  u=0 & \mbox{in}\ \R^N\backslash\Omega.
\end{cases}
\end{equation}
The natural variational approach to study problem \eqref{classic-eq1} is to consider critical points of the non-homogeneous Rayleigh type quotient
\begin{equation}\label{non-homogeneous-Rayleigh-quotient}
 R_s(u)= \frac{\displaystyle\int_{\R^N}\int_{\R^N}\displaystyle\frac{|u(x)-u(y)|^{p(x,y)}}{|x-y|^{N+sp(x,y)}}dxdy}{\displaystyle
 \int_{\Omega}|u(x)|^{\bar{p}(x)}dx}.
\end{equation}

When looking in the literature, no prior studies have delved into the exploration of problems associated with the fractional $p(x)-$Laplacian, offering a comprehensive examination of its spectrum. However, we posit that addressing this issue may not be overly challenging, as a parallel investigation to that of the $p(x)-$Laplacian could be undertaken. Nevertheless, a notable complication arises in tackling the eigenvalue problem \eqref{classic-eq1} due to its lack of homogeneity. This difficulty manifests particularly when dealing with operators featuring nonstandard growth. Specifically, eigenvalues in equation \eqref{non-homogeneous-Rayleigh-quotient} are contingent on the magnitude of $\int_{\Omega}|u(x)|^{\bar{p}(x)}dx=\mu,$ wherein distinct constraints may yield disparate eigenvalues.

An innovative approach, as proposed by Franzina and Lindqvist in \cite{Franzina-Lindqvist}, involves incorporating norms into the Rayleigh-type quotient instead of modulars. The overarching objective is to introduce homogeneity into the analysis, thereby mitigating the challenges associated with nonstandard growth operators. The focus of their investigation centered on the local $p(x)-$Laplacian. In contrast, we introduce an exploration into a nonlocal, homogeneous Rayleigh-type quotient, expanding the scope of the study. We propose the study of the following quotient:
\begin{equation}\label{homogeneous-Rayleigh-quotient}
  H_s(u):=\frac{K(u)}{k(u)},
\end{equation}
where $K(u)=[u]_{s,p(.,.)}$ and $k(u)=\|u\|_{\bar{p},\Omega}$ (Refer to Section 2 for definitions).
Notice that the homogeneity of the quotient in \eqref{homogeneous-Rayleigh-quotient} implies that eigenvalues are independent of the energy level chosen to perform the minimization, we can assume $\|u\|_{\bar{p},\Omega}=1.$\\

We are led to consider the following operator
$$
\mathcal{L}u(x):= P.V.\int_{\mathbb{R}^N}\frac{|u(x)-u(y)|^{p(x,y)-2}(u(x)-u(y))}{K(u)^{p(x,y)-1}|x-y|^{N+sp(x,y)}}dy,\ \text{for all}\ x\in\Omega.
$$
We consider the homogeneous counterpart of \eqref{classic-eq1}:
\begin{equation}\label{eq1}
  \begin{cases}
  \mathcal{L}u=\Lambda \left|\frac{u}{k(u)}\right|^{\bar{p}-2}\frac{u}{k(u)} & \mbox{in } \Omega \\
  u=0 & \mbox{in}\ \R^N\backslash\Omega.
  \end{cases}
\end{equation}

{\bf The Euler-Lagrange equation:}\\

We define
$$
\Lambda_1=\inf_{v\in X\backslash\{0\}} H_s(u).
$$

\begin{prop}\label{minimum}
  There exists a non-negative minimizer $u\in X\backslash\{0\},$ of $\Lambda_1$.
\end{prop}

\begin{proof}
  The Sobolev inequality \eqref{Sobolev-inequality} shows that $\Lambda_1>0.$ Let $(v_n)_{n\in\mathbb{N}}$ be a minimizing sequence of $\Lambda_1$ such that $ \|v_n\|_{\bar{p},\Omega}=1.$ Then
  $$
  \Lambda_1=\lim_{n\to\infty}[v_n]_{s,p(.,.)}.
  $$
  By Theorem \ref{compact-embedding}, we can extract a subsequence still denoted $(v_n)_{n\in\mathbb{N}}$ and find a function $u\in X$ such that
  $$
  v_n \to u\quad\text{strongly in}\quad L^{\bar{p}(.)}(\Omega),
  $$
  $$
  v_n \rightharpoonup u\quad\text{weakly in}\quad X.
  $$
  Since the norm is weakly sequentially lower semi-continuous, we obtain
  $$
   \frac{[u]_{s,p(.,.)}}{\|u\|_{\bar{p},\Omega}}\leq \liminf_{n\to\infty}\frac{[v_n]_{s,p(.,.)}}{\|v_n\|_{\bar{p},\Omega}}=\Lambda_1.
   $$
   Therefore, $u$ serves as a minimizer, and the same holds true for $|u|$. With this, the proof is complete.

\end{proof}

To obtain the Euler-Lagrange equation for the minimizer, we select an arbitrary test function $\varphi\in C_0^\infty(\Omega)$. Next, we examine the competing function $v_\varepsilon(x) = u(x) + \varepsilon \varphi(x)$. We observe that a necessary condition for inequality
$$
\frac{K(u)}{k(u)}\leq \frac{K(v_\varepsilon)}{k(v_\varepsilon)},
$$
which gives that
$$
\frac{d}{d\varepsilon}\left( \frac{K(v_\varepsilon)}{k(v_\varepsilon)}\right)_{\varepsilon=0}=0.
$$
Thus the necessary condition reads
\begin{equation}\label{necessary condition}
\frac{1}{K(u)}\left( \frac{d}{d\varepsilon} K(v_\varepsilon)\right)_{\varepsilon=0}=
\frac{1}{k(u)}\left( \frac{d}{d\varepsilon} k(v_\varepsilon)\right)_{\varepsilon=0}.
\end{equation}
Before proceeding with our studies, let us compute the preceding identities.
\begin{prop}\label{differentiation}
Let $u\in X$ and $\varphi\in C_0^\infty(\Omega)$. Then,
  $$
  \left( \frac{d}{d\varepsilon} K(v_\varepsilon)\right)_{\varepsilon=0}=\frac{\int_{\R^N}\int_{\R^N} \left|\frac{\nabla_s u}{K(u)}\right|^{p(x,y)-2}\frac{\nabla_s u}{K(u)}\nabla_s \varphi \frac{dxdy}{|x-y|^N}}{\int_{\R^N}\int_{\R^N}\left|\frac{\nabla_s u}{K(u)}\right|^{p(x,y)}\frac{dxdy}{|x-y|^N}}
  $$
  and
  $$
  \left( \frac{d}{d\varepsilon} k(v_\varepsilon)\right)_{\varepsilon=0}=\frac{\int_{\Omega}\left|\frac{u}{k(u)}\right|^{\bar{p}(x)-2}\frac{u}{k(u)}\varphi dx}{\int_{\Omega}\left|\frac{u}{k(u)}\right|^{\bar{p}(x)}dx},
  $$
   where
  $\nabla_s u:=\frac{u(x)-u(y)}{|x-y|^s}.$
\end{prop}

\begin{proof}
The proof for the derivation of $k$ has been established in \cite[Lemma A.1]{Franzina-Lindqvist}, and now let's address the derivation for $K.$
  Observe that for all $a,b\in \R$ and all $p>1$ we have that
  $$
  |b|^p-|a|^p=\int_{0}^{1}\frac{d}{dt}|a+t(b-a)|^pdt=p(b-a)\int_{0}^{1}|a+t(b-a)|^{p-2}[a+t(b-a)]dt.
  $$
  By the definition of the Luxemburg norm, we know that
  $$
  \int_{\R^N}\int_{\R^N}\left|\frac{\nabla_s v_\varepsilon}{K(v_\varepsilon)}\right|^{p(x,y)}\frac{dxdy}{p(x,y)|x-y|^N}=1,\quad \int_{\R^N}\int_{\R^N}\left|\frac{\nabla_s u}{K(u)}\right|^{p(x,y)}\frac{dxdy}{p(x,y)|x-y|^N}=1.
  $$
  Subtracting the preceding two expressions and employing the variables $b$ for $\frac{\nabla_s v_\varepsilon}{K(v_\varepsilon)},$ $a$ for $\frac{\nabla_s u}{K(u)}$ and $p=p(x,y)$  results in:
  \begin{equation}\label{Ax}
    0=\int_{\R^N}\int_{\R^N}\left[\frac{\nabla_s v_\varepsilon}{K(v_\varepsilon)}-\frac{\nabla_s u}{K(u)}\right]A(x,\varepsilon)\frac{dxdy}{|x-y|^N},
  \end{equation}
  where
  $$
  A(x,\varepsilon)=\int_{0}^{1}\left|\frac{\nabla_s u}{K(u)}+t\left(\frac{\nabla_s v_\varepsilon}{K(v_\varepsilon)}-\frac{\nabla_s u}{K(u)}\right)\right|^{p(x,y)-2}\left(\frac{\nabla_s u}{K(u)}+t\left(\frac{\nabla_s v_\varepsilon}{K(v_\varepsilon)}-\frac{\nabla_s u}{K(u)}\right)\right)dt.
  $$
  By employing the definition of $v_\varepsilon$, we express \eqref{Ax} in the following manner:
  $$
  \int_{\R^N}\int_{\R^N}\left[\frac{\nabla_s u}{K(v_\varepsilon)}-\frac{\nabla_s u}{K(u)}\right]A(x,\varepsilon)\frac{dxdy}{|x-y|^N}=
  -\varepsilon\int_{\R^N}\int_{\R^N}\frac{\nabla_s \varphi}{K(v_\varepsilon)}A(x,\varepsilon)\frac{dxdy}{|x-y|^N},
  $$
  from where
  \begin{equation}\label{Kepsilon}
  \frac{K(v_\varepsilon)-K(u)}{\varepsilon}\int_{\R^N}\int_{\R^N}\frac{\nabla_s u}{K(u)}A(x,\varepsilon)\frac{dxdy}{|x-y|^N}=\int_{\R^N}\int_{\R^N}\nabla_s \varphi A(x,\varepsilon)\frac{dxdy}{|x-y|^N}.
  \end{equation}
  Since $K(v_\varepsilon)$ depends continuously on $\varepsilon$, then
  $$
  \frac{\nabla_s v_\varepsilon}{K(v_\varepsilon)}-\frac{\nabla_s u}{K(u)}\to 0\ \text{a.e. as}\ \varepsilon\to 0^+,
  $$
  and we deduce that
  $$
  A(x,\varepsilon)\to \left|\frac{\nabla_s u}{K(u)}\right|^{p(x,y)-2}\frac{\nabla_s u}{K(u)}\ \text{a.e. as}\ \varepsilon\to 0^+.
  $$
  Letting $\varepsilon\to 0^+$ in \eqref{Kepsilon}, we get
  $$
  \left( \frac{d}{d\varepsilon} K(v_\varepsilon)\right)_{\varepsilon=0}\int_{\R^N}\int_{\R^N}\left|\frac{\nabla_s u}{K(u)}\right|^{p(x,y)}\frac{dxdy}{|x-y|^N}=\int_{\R^N}\int_{\R^N} \left|\frac{\nabla_s u}{K(u)}\right|^{p(x,y)-2}\frac{\nabla_s u}{K(u)}\nabla_s \varphi \frac{dxdy}{|x-y|^N}.
  $$
  That is
  $$
  \left( \frac{d}{d\varepsilon} K(v_\varepsilon)\right)_{\varepsilon=0}=\frac{K\int_{\R^N}\int_{\R^N} \left|\frac{\nabla_s u}{K(u)}\right|^{p(x,y)-2}\frac{\nabla_s u}{K(u)}\frac{\nabla_s \varphi}{K(u)} \frac{dxdy}{|x-y|^N}}{\int_{\R^N}\int_{\R^N}\left|\frac{\nabla_s u}{K(u)}\right|^{p(x,y)}\frac{dxdy}{|x-y|^N}}
  $$
\end{proof}

Now, let's advance toward the primary objective of this part, which is to deduce the Euler-Lagrange equation that characterizes the minimum of $\Lambda_1$. By combining the Proposition \ref{differentiation} and the identity labeled as \eqref{necessary condition} , it becomes clear that:
\begin{equation}\label{Lambda1}
\int_{\R^N}\int_{\R^N} \left|\frac{\nabla_s u}{K(u)}\right|^{p(x,y)-2}\frac{\nabla_s u}{K(u)}\nabla_s \varphi \frac{dxdy}{|x-y|^N}=
\Lambda_1 S(u)\int_{\Omega}\left|\frac{u}{k(u)}\right|^{\bar{p}(x)-2}\frac{u}{k(u)}\varphi dx,\quad\forall\varphi\in C_0^\infty(\Omega),
\end{equation}
where
\begin{equation}\label{S(u)}
S(u):=\frac{\int_{\R^N}\int_{\R^N}\left|\frac{\nabla_s u}{K(u)}\right|^{p(x,y)}\frac{dxdy}{|x-y|^N}}{\int_{\Omega}\left|\frac{u}{k}\right|^{\bar{p}(x)}dx}.
\end{equation}
The preceding characterization presents the weak form of problem \eqref{eq1}, and subsequently, we introduce the following definition.
\begin{defn}\label{def-eigenvalue}
  We say that $\Lambda$ is an eigenvalue of \eqref{eq1} with eigenfunction $u\in X$ if $u$ is a nontrivial weak solution to \eqref{eq1}, that is,
  \begin{equation}\label{Lambda}
\int_{\R^N}\int_{\R^N} \left|\frac{\nabla_s u}{K(u)}\right|^{p(x,y)-2}\frac{\nabla_s u}{K(u)}\nabla_s \varphi \frac{dxdy}{|x-y|^N}=
\Lambda S(u)\int_{\Omega}\left|\frac{u}{k(u)}\right|^{\bar{p}(x)-2}\frac{u}{k(u)}\varphi dx,\quad\forall\varphi\in X.
\end{equation}
\end{defn}

\begin{thm}\label{first-eigenvalue}
$\Lambda_1$ is the first eigenvalue and the minimum $u$ of $\Lambda_1$ which is nontrivial is the corresponding first eigenfunction.
\end{thm}

\begin{proof}
In light of identity \eqref{Lambda1}, it becomes clear that $\Lambda_1$ serves as an eigenvalue for \eqref{eq1}, with $u$ representing the corresponding eigenfunction. Now, let's verify that $\Lambda_1$ holds the position of the first eigenvalue. Let $\Lambda$ be an eigenvalue of \eqref{eq1} and $v$ its corresponding eigenfunction. By setting $\varphi=v = u$ in equation \eqref{Lambda} and considering the definition of $S(v)$ as given in \eqref{S(u)}, we can derive:
\begin{align*}
   K(v)\int_{\R^N}\int_{\R^N}\left|\frac{\nabla_s v}{K(v)}\right|^{p(x,y)}\frac{dxdy}{|x-y|^N}&=\Lambda S(v)k(v) \int_{\Omega}\left|\frac{v}{k}\right|^{\bar{p}(x)}dx\\
   &=\Lambda k(v)\int_{\R^N}\int_{\R^N}\left|\frac{\nabla_s v}{K(v)}\right|^{p(x,y)}\frac{dxdy}{|x-y|^N}.
\end{align*}
Then
$$
\Lambda=\frac{K(v)}{k(v)}\geq\Lambda_1.
$$
Hence $\lambda_1$ is the first eigenvalue.

\end{proof}

{\bf Sequence of eigenvalues:}\\

Put
$$
\mathcal{M}:=\left\{u\in X:\ k(u)=1\right\}.
$$
The first eigenvalue can be defined as
$$
\Lambda_1=\inf_{u\in \mathcal{M}}K(u).
$$
$\mathcal{M}$ is a submanifold of class $C^1$ in $X.$

\begin{lema}
Let $u\in X\backslash\{0\}.$ For all $v\in X$ the following inequalities hold
\begin{equation}\label{Kk}
  \left|\langle K^{'}(u),v\rangle\right|\leq K(v)\quad\text{and}\quad\left|\langle k^{'}(u),v\rangle\right|\leq k(v).
\end{equation}
\end{lema}

\begin{proof}
 The second inequality of \eqref{Kk} has been established in \cite[Proposition 3.6]{Alves et al}. Employing analogous techniques, we extend this proof to establish the first one.
\end{proof}

\begin{prop}\label{PS}
  The functional $\tilde{K}|_{\mathcal{M}}$ satisfies the $(PS)$ condition, i.e., every sequence $(u_k)\subset \mathcal{M}$ such that $\tilde{K}(u_k)\to c$ for some $c\in\R$ and $\tilde{K}^{'}(u_k)\to 0$ in $X^{'}$ admits a convergent subsequence.
\end{prop}

\begin{proof}
  Let  $c\in \R$ and $(c_k)\subset\R$ a sequence such that
  \begin{equation}\label{ck}
    K(u_k)\to c\quad\text{and}\quad K^{'}(u_k)-c_k k^{'}(u_k)\to 0\ \text{in}\ X^{'}.
  \end{equation}
  On one hand, we can see that
  $$
  \langle K^{'}(u_k)-c_k k^{'}(u_k),u_k\rangle=\langle K^{'}(u_k),u_k\rangle -c_k\langle k^{'}(u_k),u_k\rangle=K(u_k)-c_k.
  $$
  On the other hand
  $$
  \left|\langle K^{'}(u_k)-c_k k^{'}(u_k),u_k\rangle\right|\leq \| K^{'}(u_k)-c_k k^{'}(u_k)\|_{X^*}K(u_k) \to 0.
  $$
  Then $c_k\to c.$ Since $(u_k)$ is bounded in $X$, up to subsequence, $u_k\rightharpoonup u$ in $X$ and $u_k\to u$ in $L^{\bar{p}(x)}(\Omega).$ By \eqref{Kk}, we get
  $$
  \left|\langle k^{'}(u_k),u_k-u\rangle\right|\leq k(u_k -u)=\|u_k-u\|_{\bar{p},\Omega}\to 0.
  $$
  from this and \eqref{ck} we deduce that
  \begin{equation*}
    \langle K^{'}(u_k),u_k-u\rangle\to 0.
  \end{equation*}
  Since
  $$
   \langle K^{'}(u_k),u_k-u\rangle=K(u_k)- \langle K^{'}(u_k),u\rangle\geq K(u_k)-K(u),
  $$
  then, by the weak lower semicontinuity of the norm, we get
  $$
  \limsup_{k\to\infty} K(u_k)\leq K(u)\leq \liminf_{k\to\infty}K(u_k).
  $$
\end{proof}

As $\mathcal{M}$ is a closed symmetric submanifold with $C^1$ regularity in $X$, the previously mentioned result empowers us to derive a sequence of eigenvalues for equation \eqref{eq1} by employing a minimax procedure.

\begin{defn}
  For $n\in\mathbb{N},$ we define the $n-$th variational eigenvalue $\Lambda_n$ of \eqref{eq1} as
  \begin{equation}\label{sequence}
    \Lambda_n:=\inf_{A\in \mathcal{C}_m}\sup_{u\in A} \tilde{K}(u),
  \end{equation}
  where  $\mathcal{C}_m:=\left\{C \subset \mathcal{M}: C\right.$ is compact, $\left.C=-C, \gamma(C) \geq m\right\}$ and $\gamma$ is the Krasnoselskii genus.
\end{defn}

\begin{prop}
  The sequence $(\Lambda_n)$ is nondecreasing, divergent, and for every $n\geq 1$ there exists $u_n\in \mathcal{M}$ solving \eqref{eq1}, with $\Lambda=\Lambda_n.$
\end{prop}

\begin{proof}
  Under Proposition \ref{PS}, we can employ \cite[Corollary 4.1]{Szulkin} to establish that the values specified in \eqref{sequence} indeed constitute eigenvalues of \eqref{eq1} as per Definition \ref{def-eigenvalue}. Additionally, given that $C_{m+1}\subset C_m$ holds for all $m,$ the sequence $(\Lambda_n)$ is characterized by a divergent, nondecreasing pattern.
\end{proof}



\begin{thebibliography}{99}

\bibitem{Adams} R. A. Adams, Anisotropic Sobolev inequalities, \v{C}asopis pro p\v{e}stov\'{a}n\'{\i} matematiky,  113 (1988), No. 3, 267-279.

\bibitem{Benedek-Panzone} A. Benedek and R. Panzone, The spaces $L^p$ with mixed norm, Duke Math. J. 28 (1961), 301-324.





\bibitem{Fan1} X. Fan, Anisotropic variable exponent Sobolev spaces and $\overrightarrow{p(x)}$Laplacian equations, {\it Complex Variables and Elliptic Equations,} 56(7-9), 623-642, 2011.

\bibitem{Antontsev-Shmarev1} S. Antontsev, S. Shmarev, Elliptic equations with anisotropic nonlinearity and nonstandard growth conditions, Handbook of Differential Equations: Stationary Partial Differential Equations 3, 2006,  1-100.

\bibitem{Antontsev-Shmarev2} S. Antontsev, S. Shmarev, Parabolic Equations with Anisotropic Nonstandard Growth Conditions,  {\it Int. Ser.  Numer. Math.} 154, 2006, 33-44.

\bibitem{MPR} M. Mih\u{a}ilescu, P. Pucci, V. R\u{a}dulescu, Eigenvalue problems for anisotropic quasilinear elliptic equations with variable exponent, {\it J. Math. Anal. App.} 340, 2008, 687-698.





\bibitem{DHHR} L. Diening, P. Harjulehto, P. H\"{a}st\"{o} and M. Ruzicka, Lebesgue and Sobolev spaces with variable exponents,
    Lecture Notes in Mathematics, vol. 2017, Springer-Verlag, Heidelberg, 2011.

\bibitem{Radulescu-Repovs1} V.D. R\u{a}dulescu and D. Repovs, Partial Differential Equations with Variable Exponents: Variational Methods and Qualitative Analysis, CRC Press, Taylor \& Francis Group, Boca Raton FL, 2015.

\bibitem{Fan-Zhao} X. Fan and D. Zhao, On the spaces $L^{p(x)}(\Omega)$ and $W^{m,p(x)}(\Omega)$, {\it J. Math. Anal. Appl.} {\bf 263} (2001), 424-446.

\bibitem{Kova}  O. Kov\'{a}\v{c}ik and J. R\'{a}kosn\'{\i}k, On spaces $L^{p(x)}$ and $W^{m,p(x)}$, {\it Czechoslovak Math. J}. {\bf41} (1991), 592-618.



\bibitem{NPV} E. Di Nezza, G. Palatucci, and E. Valdinoci, Hitchhiker's guide to the fractional Sobolev spaces, {\it Bull. Sci. Math.,} 136(5), 521-573, 2012.

\bibitem{Caffarelli-Silvestre} L. Caffarelli, L. Silvestre, An extension problem related to the fractional Laplacian, {\it Communications in
partial differential equations,} 32(8), 1245-1260, 2007.

\bibitem{BBM1} J. Bourgain, H. Brezis, P Mironescu, Another look at Sobolev spaces, {\it Optimal control and partial differential equations}, IOS, Amsterdam, 439-455. MR 3586796 2001.

\bibitem{BBM2} J. Bourgain, H. Brezis, P. Mironescu, Limiting embedding theorems for $W^{ s,p}$ when $s\uparrow1$ and applications, {\it J. Anal. Math,} 87, 77-101, 2002.






\bibitem{Ceresa-Bonder1} I. Ceresa Dussel, J. Fern\'{a}ndez Bonder, A Bourgain-Brezis-Mironescu formula for anisotropic fractional Sobolev spaces and applications to anisotropic fractional differential equations, {\it  J. Math. Anal. Appl.} 519(2), 126805, 2023.

\bibitem{Ceresa-Bonder2}  I. Ceresa Dussel, J. Fernández Bonder, Existence of Eigenvalues for Anisotropic and Fractional Anisotropic Problems via Ljusternik-Schnirelmann Theory, 	arXiv:2309.14301.

\bibitem{CKW1} J. Chaker, M. Kim, M. Weidner, The concentration-compactness principle for the nonlocal anisotropic $p-$Laplacian of mixed order, {\it Nonlinear Analysis,} 232,  113254, 2023.










\bibitem{Kaufmann} U. Kaufmann, J.D. Rossi and R. Vidal, Fractional Sobolev spaces with variable exponents and fractional $p(x)$-Laplacians,  {\it Electron. J. Qual. Theory Differ. Equ}. \textbf{76} (2017), 1-10.

\bibitem{Ho-Kim}  K. Ho, Y.H. Kim, A priori bounds and multiplicity of solutions for nonlinear elliptic problems involving the fractional $p(\cdot)-$Laplacian, {\it Nonlinear Anal.} {\bf 188} (2019), 179-201.

\bibitem{Kim} M. Kim, Bourgain, Brezis and Mironescu theorem for fractional Sobolev spaces with variable exponents, {\it Annali di Matematica Pura ed Applicata, } 202, 2653-2664, 2023.

\bibitem{anouar1} A. Bahrouni, Comparison and sub-supersolution principles for the fractional $p(x)-$Laplacian, {\it J. Math. Anal. Appl}. {\bf 458}
	(2018), 1363-1372.

\bibitem{anouar2} A. Bahrouni and V. Radulescu, On a new fractional Sobolev space and application to nonlocal variational problems with
	variable exponent, {\it Discrete Contin. Dyn. Syst. Ser. S.} {\bf 11} (2018), 379-389.

\bibitem{Bahrouni-Ho} A. Bahrouni and K. Ho, Remarks on Eigenvalue Problems for Fractional $p(.)-$Laplacian, {\it Asymptotic Analysis,} {\bf 123}(1-2) (2021), 139-156.








\bibitem{Franzina-Lindqvist} G. Franzina and P. Lindqvist, An eigenvalue problem with variable exponent, {\it Nonlinear Analysis: Theory, Methods $\&$ Applications,} {\bf 85} (2013), 1-16.

\bibitem{Alves et al} C.O. Alves, G. Ercole and M.D. Huam\'{a}n Bola\={o}s, Minimization of quotients with variable ewponents, {\it J. Differential Equations,} {\bf 265} (2018), 15996-1626.

\bibitem{Szulkin} A. Szulkin, Ljusternik Schnirelmann theory on $C^1$ manifolds, {\it Ann. Inst. H. Poincar\'{e} Anal. Non Lin\'{e}aire,} {\bf 5} (1988), 119-139.

\bibitem{Bellido-Mora}
J. C. Bellido and C. Mora-Corral,
``Existence for Nonlocal Variational Problems in Peridynamics,''
\textit{SIAM Journal on Mathematical Analysis},
vol. 46, no. 1, pp. 890-916, 2014.

\bibitem{Bellido-Ortega}Bellido, J.C., Ortega, A. Spectral Stability for the Peridynamic Fractional p-Laplacian. Appl Math Optim 84 (Suppl 1), 253–276 (2021). https://doi.org/10.1007/s00245-021-09768-6.

\bibitem{Ponce}
Ponce, A. (2004). A new approach to Sobolev spaces and connections to $\Gamma$-convergence. \textit{Calculus of Variations}, 19, 229-255. doi:10.1007/s00526-003-0195-z.


\end{thebibliography}
\end{document}